\documentclass{article}
\usepackage{amsmath}
\usepackage{amssymb}
\usepackage{esint}
\usepackage{geometry}
\usepackage{cleveref}

\makeatletter
 \usepackage{amsfonts}
\usepackage{color}
\usepackage{graphicx}
\providecommand{\U}[1]{\protect \rule{.1in}{.1in}}

\newtheorem{theorem}{Theorem}[section]

\newtheorem{corollary}[theorem]{Corollary}

\newtheorem{definition}[theorem]{Definition}
\newtheorem{example}[theorem]{Example}

\newtheorem{lemma}[theorem]{Lemma}

\newtheorem{proposition}[theorem]{Proposition}
\newtheorem{remark}[theorem]{Remark}

\newenvironment{proof}[1][Proof]{\noindent \textbf{#1.} }{\  \rule{0.5em}{0.5em}}

\def\lu{\underline{\mu}}
\def\ou{\overline{\mu}}

\def\lt{\left}
\def\rt{\right}

\def\br{\mathbb{R}}

\def\cp{\mathcal{P}}
\def\cf{\mathcal{F}}

\def\cf{\mathcal{F}}

\def\pe{{\mathbb{E}^{\cp}}}

\def\bbl{\boldsymbol{\lambda}}
\def\bbm{\boldsymbol{\mu}}

\def\bbn{\boldsymbol{\nu}}
\def\bbk{\boldsymbol{\kappa}}

\def\ccp{{\rm{co}(\cp)}}

\def\ccs{{\rm{cov}{(\bf{\Sigma}})}}

\makeatother

\begin{document}

\title{On the upper and lower covariances under multiple probabilities
\footnote{This work was supported by NSF of Shandong Province (No.ZR2021MA018),  NSF of China (No.12326603 and No.11601281),  National Key R\&D Program of China (No.2018YFA0703900) and the Qilu Young Scholars Program of Shandong University.}}

\author{Xinpeng Li\footnote{Research Center for Mathematics and Interdisciplinary Sciences; Frontiers Science Center for Nonlinear Expectations (Ministry of Education), Shandong University, 266237, Qingdao, China.
  Email:lixinpeng@sdu.edu.cn} \ \ \
 Jingxu Niu\footnote{Research Center for Mathematics and Interdisciplinary Sciences, Shandong University, 266237, Qingdao, China.}  \ \ \
 Ke Zhou\footnote{Business School of Hunan University, 410012, Changsha, China.}
}
\date{ }

\maketitle

\abstract{In this paper, we define the upper (resp. lower) covariance under multiple probabilities via a corresponding max-min-max (resp. min-max-min) optimization problem and the related properties of covariances are obtained. In particular, we propose a fast algorithm of calculation for upper and lower covariances under the finite number of probabilities. As an application, our algorithm can be used to solve a class of quadratic programming problem exactly, and we obtain a probabilistic representation of such quadratic programming problem. }



\section{Introduction}
\par In the classical probability theory, covariance is a measure of the joint variability of two random variables. Precisely, let $(X,Y)$ be two random variables on the probability space  $(\Omega,\mathcal{F},P)$, the covariance of $X$ and $Y$ under $P$, denoted by $C_P(X,Y)$, is defined by
\begin{equation}\label{cov}
C_P(X,Y)=E_P[(X-E_P[X])(Y-E_P[Y])],
\end{equation}
where $E_P$ is the  mathematical expectation introduced by $P$. Such notion was widely used in many fields, such as financial economics (see for example, Markowitz \cite{markowitz}), genetics and molecular biology (see for example, Price \cite{price1970}),  etc. In particular, if $Y=X$ in \cref{cov}, the corresponding covariance is called variance of $X$, denoted by $V_P(X)$, which is used to describe the fluctuation of a random variable.

The aim of this paper is to generalize the notion of covariance in the case of multiple probability measures. The motivation is inspired by a simple example. Let $(X,Y)$ be two random variables representing the daily returns of two stocks in the stock market. We hope to estimate the ``risk" (variance and covariance) of the stock in the next month. For simplicity,  we assume there are only two scenarios of the stock markets, bull market and bear market. The returns of the stocks in bull and bear markets are  characterized by two probabilities respectively, and we call them Bull probability and Bear probability. We do not know the true state (bull or
bear) of the stock market in the future. This is a typical example of Knightian uncertainty, introduced by Knight \cite{knight2013risk}.
  We prefer the bivariate  normal distribution to characterize the ``risk'' in the bull and bear markets. The corresponding parameters can be estimated from the historical  data of the stocks in the bull and bear market respectively. More precisely, $(X,Y)\sim N(0.1,0.1,0.4,0.4,0.1)$ under Bull probability and $(X,Y)\sim N(-0.1, -0.1,0.4,0.4,0.1)$ under Bear probability,  where $N(\mu_X,\mu_Y,\sigma_X^2,\sigma_Y^2,\rho)$ denotes the bivariate normal distribution with means $\mu_X$ and $\mu_Y$ and variances $\sigma_X^2$ and $\sigma_Y^2$ respectively, and $\rho$ is the correlation coefficient. Now the problem is that how to reasonably define the variance and covariance for such stocks with Knightian uncertainty based on the parameters for each scenarios. Sublinear expectation theory introduced by Peng (see Peng \cite{peng2010nonlinear}) is a powerful tool to deal with such problems with uncertainty, in which the sublinear expectation is a sublinear functional on the certain space of random variables. One typical example of sublinear expectation is the upper expectation $\pe$, defined by
\begin{equation*}
  \pe[X]=\sup_{P\in\cp}E_P[X],
\end{equation*}
where $\cp$ is the set of probability measures characterizing Knightian uncertainty. The notion of sublinear expectation is also known as the upper expectation in robust statistics (see Huber \cite{huber}), or the upper prevision in the theory of imprecise probabilities (see Walley \cite{walley1991statistical}), and has the closed relation with coherent risk measures (see Artzner et al. \cite{A}, F\"{o}llmer and Schied \cite{FS}).  The sublinear expectation theory has been applied successfully in data analysis to account for Knightian uncertainty (see Ji et al. \cite{JPY}, Peng et al. \cite{PYY}). We will define the notion of upper and lower covariance based on the sublinear expectation $\pe$ and study their properties systematically.

The notion of upper and lower variance was firstly introduced in Walley \cite{walley1991statistical} for the bounded random variables under coherent prevision and then generalized by Li et al. \cite{li2022upper}. Formally, the upper and lower variances of $X$, denoted by $\overline{V}(X)$ and $\underline{V}(X)$ respectively, are defined by
$$\overline{V}(X)=\min_{\mu\in\mathbb{R}}\pe[(X-\mu)^2];$$
$$\underline{V}(X)=\min_{\mu\in\br}-\pe[-(X-\mu)^2].$$
Quaeghebeur \cite{quaeghebeur2008lower} studied the notion of upper and lower covariances for bounded random variables, where the covariance is calculated by an optimization problem. This paper will further generalize the notions of covariance under the framework of sublinear expectation theory. As far as we know, there is quite little literature on this topic. Both \cite{walley1991statistical} and \cite{quaeghebeur2008lower} are not concern about the algorithm for calculation of these notions. Walley  pointed out that it is quite easy to calculate the lower variance, but
the calculation of upper variance is more difficult (see Appendix G in \cite{walley1991statistical}). The calculation of upper and lower variances under multiple probability measures is solved in Li et al. \cite{li2023}. This paper will focus on the calculation of the upper  and lower covariance under multiple probability measures. It will be very useful in practice. For example, it can be used to determine the parameters in the robust portfolio models based on the historical data in the market. Differing from lower variance, which is easily calculated, lower covariance is also challenging to compute, akin to the case of upper covariance. The primary challenge in calculating upper (or lower) covariance lies in the complexity of the max-min-max (or min-max-min) optimization problem. While upper (or lower) covariance can be determined through quadratic programming using the envelope theorem, it becomes a hard problem in certain instances due to the indefinite nature of the matrix in the quadratic term. We present a more efficient algorithm that operates in polynomial time for calculating both upper and lower covariances. This algorithm can be seen as an application of our findings in quadratic programming.

  \par The paper is structured as follows:

In Section 2, we revisit the concepts of sublinear expectation, upper variance and lower variance.
Section 3 is dedicated to a systematic exploration of the definitions and properties associated with upper and lower covariances.
The methodology for calculating upper and lower covariances is detailed in Section 4.
In Section 5, we illustrate an application of our findings within the domain of quadratic programming.

\section{Preliminaries}

We firstly recall the preliminaries of sublinear expectation theory introduced in Peng  \cite{peng2010nonlinear}, which is a powerful tool to deal with problems with uncertainty, i.e., problems under multiple probabilities.

Let $\Omega$ be a given set and  $(\Omega,\cf)$ be a measurable space. Let $\cp$ be a set of probability measures on $(\Omega,\cf)$ characterizing Knightian uncertainty. We define the corresponding upper expectation $\pe$ by
$$\pe[X]=\sup_{P\in\cp}E_P[X].$$
Obviously, $\pe$ is a sublinear expectation satisfying
  \begin{itemize}
    \item [(1)]Monotonicity: $\pe[X]\le \pe[Y],$ if $X\le Y$;
    \item [(2)]Constant preserving : $\pe[c]=c, \forall c\in \mathbb{R}$;
    \item [(3)]Sub-additivity : $\pe[X+Y]\le \pe[X]+\pe[Y]$;
    \item [(4)]Positive homogeneity : $\pe[\lambda X]=\pe[X], \forall \lambda \geq 0$.
  \end{itemize}
We call $(\Omega,\cf,\pe)$ the sublinear expectation space.

In this paper,  we denote $\ccp$ the convex hull of $\cp$. It is clear that
  $$\pe[X]=\sup_{P\in\cp}E_P[X]=\sup_{P\in\ccp}E_P[X].$$

The {{upper expectation} }of a random variable $X$, denoted by $\overline{\mu}_X$, is represented by $\pe[X]$. Similarly, the {{lower expectation} }of $X$, denoted by $\underline{\mu}_X$, is given by $-\pe[-X]$. These are respectively referred to as the {upper mean} and {lower mean} of $X$. The {{mean-uncertainty} }of a random variable $X$, denoted by $M_X$, is characterized by the interval $[\underline{\mu}_X,\overline{\mu}_X]$. This interval represents the bounds within which the mean of $X$ is expected to vary.

We begin by revisiting the concept of upper and lower variances, initially introduced by Walley \cite{walley1991statistical} for bounded random variables under coherent prevision. This notion was subsequently extended and generalized by Li et al. \cite{li2022upper}.

Let  $X$ be a random variable on the probability space $(\Omega,\cf,P)$ with finite second moment. The variance of $X$, denoted by $V_P(X)$, is defined by
 $$V_P(X)=E_P[(X-E_P[X])^2].$$
 We note that
 \begin{equation}\label{eq1}
 E_P[(X-\mu)^2]=V_P(X)+(E_P[X]-\mu)^2,
 \end{equation}
immediately,
 $$V_P(X)=\min_{\mu\in\br}E_P[(X-\mu)^2],$$
then we use $\pe$ instead of $E_P$ to obtain the following definition.

\begin{definition}\label{def21}
    For each random variable $X$ on the sublinear expectation space $(\Omega,\mathcal{F},\pe)$ with $\pe[X^2]
    <\infty$, we define the {upper variance }of $X$ as
    \[\overline{V}(X): = \min_{\mu\in M_X} \pe[(X-\mu)^2]
     \]
     and the {lower variance} of $X$ as
     \[ \underline{V}(X): = \min_{\mu\in M_X} \lt(-\pe[-(X-\mu)^2]\rt)
       \]
 \end{definition}

Since $\pe[(X-\mu)^2]$ is a strict convex function of $\mu$, there exists a unique $\mu^*\in M_X$ such that $\overline{V}(X)=\pe[(X-\mu^*)^2]$.  It is not hard to show that the domain of $\mu$ can be enlarged from $M_X$ to $\br$ in Definition \ref{def21}.

 Thanks to the minimax theorem in Sion \cite{sion1958general}, we have the following variance envelop theorem.
\begin{theorem}\label{vpt}
    Let $X$ be a random variable on sublinear expectation space $(\Omega,\mathcal{F},\pe)$ with
    $\pe[X^2]<\infty$. Then we have
    \begin{itemize}
      \item[(1)] $\overline{V}(X)=\sup_{P\in\ccp}V_P(X)$.
      \item[(2)] $\underline{V}(X)=\inf_{P\in\ccp}V_P(X)=\inf_{P\in\cp}V_P(X)$.
      \end{itemize}
\end{theorem}

\begin{remark}
Unlike the envelope theorems presented in Walley \cite{walley1991statistical} and Li et al. \cite{li2022upper}, our methodology does not require the weak compactness of $\cp$. Yet, it is important to note that the convexity of $\cp$ plays a crucial role, as demonstrated in the example that follows.
\end{remark}

\begin{example}\label{em1}
Let $X$ be normally distributed with $X\sim N(0.1,0.4)$ under $P_1$ and $X\sim N(-0.1,0.4)$ under $P_2$. Taking $\cp=\{P_1,P_2\}$, we obtain $\underline{V}(X)=0.4$ and
$$\overline{V}(X)=0.41>0.4=\max_{P\in\cp}V_P(X).$$
\end{example}

Recently, Li et al. \cite{li2023} obtain the following proposition for upper variance under finitely many probabilities.

\begin{proposition}\label{pp25}
Let $\cp=\{P_1,\cdots,P_K\}$ and $X$ be a random variable with $\pe[X^2]<\infty$.  Then we have
$$\overline{V}(X)=\max_{1\leq i<j\leq K}\overline{V}_{ij}(X),$$
where $\overline{V}_{ij}(X)$ is the upper variance of $X$ under probabilities $P_i$ and $P_j$.

Furthermore, let $\mu_i=E_{P_i}[X]$ and $\kappa_i=E_{P_i}[X^2]$, $1\leq i\leq K$.
Then we have
$$\overline{V}_{ij}(X)=\max\{\kappa_i-\mu_i^2,\kappa_j-\mu_j^2,h(\mu_{ij})\},$$
where  $h(x)=x^2-2\mu_ix+\kappa_i$ and
\begin{align*}
 \mu_{ij}=
 \begin{cases}
    \lt(\underline{\mu}_{ij}\vee\frac{\kappa_{j}-\kappa_{i}}{2(\mu_{j}-\mu_{i})}\rt)\wedge\overline{\mu}_{ij}, & \mu_i\neq\mu_j, \\
    \mu_i, & \mu_i=\mu_j,
 \end{cases}
\end{align*}
with $\overline{\mu}_{ij}=\mu_i\vee\mu_j$ and $\underline{\mu}_{ij}=\mu_i\wedge\mu_j$.
\end{proposition}

\section{Upper and Lower Covariances}

In this section, we comprehensively examine the definitions and properties of upper and lower covariances, extending  the concepts of upper and lower variances discussed in the previous section.

In alignment with the optimization-based characterization of variance, covariance can be similarly conceptualized as an optimization problem. Indeed, for any real numbers $\mu_i\in\mathbb{R},\ i=1,2$ , it holds that
  \begin{align*}
    C_{P}(X,Y)&={E}_{P}[(X-{E}_{P}[X])(Y-{E}_{P}[Y])]\\
    &={E}_{P}[(X-\mu_1 +\mu_1 -{E}_{P}[X])(Y-\mu_2+\mu_2-{E}_{P}[Y])]\\
    &={E}_{P}[(X-\mu_1)(Y-\mu_2)]-(\mathbb{E}_{P}[X]-\mu_1)({E}_{P}[Y]-\mu_2).
  \end{align*}
Then we have
\[ C_P(X,Y)+(E_P[X]-\mu_1)(E_P[Y]-\mu_2)=E_P[(X-\mu_1)(Y-\mu_2)].
  \]
The second term of left side in the expression in $(\mu_1,\mu_2)$ is a saddle surface and the saddle point is $(E_P[X],E_P[Y])$.
Then the covariance of $X$ and $Y$ can be written as
\begin{align}
    C_P(X,Y)&=\min_{\mu_1\in \br}\max_{\mu_2\in \br}E_P[(X-\mu_1)(Y-\mu_2)]\label{eq_cov1}\\
    &=\max_{\mu_2\in \br}\min_{\mu_1\in \br}E_P[(X-\mu_1)(Y-\mu_2)].\label{eq_cov2}
\end{align}

\par We can now proceed to provide the formal definitions of upper and lower covariance within the framework of sublinear expectation.

\begin{definition}\label{def_cov}
  For two random variables $X$ and $Y$ on sublinear expectation space $(\Omega,\mathcal{F},\pe)$ with $\pe[X^2]+\pe[Y^2]<\infty$, we define  the upper covariance of $X$ and $Y$ as
   $$\overline{C}(X,Y):=\max_{\mu_2\in M_Y}\min_{\mu_1\in M_X}\pe[(X-\mu_1)(Y-\mu_2)],$$
   and the lower covariance of $X$ and $Y$ as
   $$\underline{C}(X,Y):=\min_{\mu_2\in M_Y}\max_{\mu_1\in M_X}\left(-\pe[-(X-\mu_1)(Y-\mu_2)]\right),$$
 where $M_X=[\underline{\mu}_X,\overline{\mu}_X]$ and $M_Y=[\underline{\mu}_Y,\overline{\mu}_Y]$.
\end{definition}

Unlike the covariance can be represented by the maximin optimization or the minimax optimization equivalently in \cref{eq_cov1} and \cref{eq_cov2}, the following examples shows that the maximin (resp. minimax) optimization for upper (resp. lower) covariance in Definition \ref{def_cov} is not equal to the minimax (resp. maximin) optimization.

\begin{example}\label{em2}
  Let $(X,Y)\sim N(-1,0,1,1,1)$ under $P_1$ and $(X,Y)\sim N(0,1,1,1,1)$ under $P_2$. We consider $\cp=\{P_1,P_2\}$ and define function $f$ as
  $$f(\mu_1,\mu_2)=\pe[(X-\mu_1)(Y-\mu_2)].$$
  By simple calculation, we obtain
  \begin{equation*}
    f(\mu_1,\mu_2)=
    \begin{cases}
      1+(\mu_1+1)\mu_2,&\mu_1+\mu_2\ge 0\\
      1+\mu_1(\mu_2-1),&\mu_1+\mu_2<0
    \end{cases}.
  \end{equation*}
  Then we have
   $$ \overline{C}(X,Y)=\max_{\mu_2\in M_Y}\min_{\mu_1\in M_X} f(\mu_1,\mu_2)=\frac{5}{4}.$$
  But it is easy to check that
  $$\min_{\mu_1\in M_X}\max_{\mu_2\in M_Y} f(\mu_1,\mu_2)=\frac{3}{2}\neq\overline{C}(X,Y).$$
\end{example}

\begin{example}\label{emem3}
  Let $(X,Y)\sim N(-1,0,1,1,1)$ under $P_1$ and $(X,Y)\sim N(0,-1,1,1,1)$ under $P_2$, we define function $g$ as
  \begin{equation*}
    g(\mu_1,\mu_2)=-\pe[-(X-\mu_1)(Y-\mu_2)]=
    \begin{cases}
      1+(\mu_1+1)\mu_2,&-\mu_1+\mu_2\le 0\\
      1+\mu_1(\mu_2+1),&-\mu_1+\mu_2>0
    \end{cases}.
  \end{equation*}
  Then we have
   $$ \underline{C}(X,Y)=\min_{\mu_2\in M_Y}\max_{\mu_1\in M_X} g(\mu_1,\mu_2)=\frac{3}{4}.$$
  On the other hand,
  $$ \max_{\mu_2\in M_Y}\min_{\mu_1\in M_X} g(\mu_1,\mu_2)=\frac{1}{2}\neq \underline{C}(X,Y).$$
\end{example}

\begin{remark}
 Due to the strict convexity of $\pe[(X-\mu)^2]$ in $\mu$, there exists a unique $\mu^*\in M_X$ such that $\overline{V}(X)=\pe[(X-\mu^*)^2]$. But such uniqueness does not hold for the case of upper covariance. In Example \ref{emem3}, there exists $(\mu_1^*,\mu_2^*)=(0,0)$ and $(-1,-1)$ such that
 $$\overline{C}(X,Y)=\pe[(X-\mu_1^*)(Y-\mu_2^*)]=1.$$
\end{remark}

A natural question is that why we choose maximin optimization as the definition of upper covariance? One of reasons is due to the following envelope theorem, which generalizes the envelope theorem for variance (see Theorem \ref{vpt}). The other reason can be found in Remark \ref{r31}.

\begin{theorem}\label{prop_rt}
  For two random variables $X$ and $Y$ on the sublinear expectation space $(\Omega,\mathcal{F},\pe)$ with $\pe[X^2]+\pe[Y^2]<\infty$, we have
  \begin{itemize}
    \item [(1)]$\overline{C}(X,Y)=\sup_{P\in\ccp}C_P(X,Y).$
    \item [(2)]$\underline{C}(X,Y)=\inf_{P\in\ccp}C_P(X,Y).$
  \end{itemize}
\end{theorem}
\begin{proof} We only prove the upper covariance case, the lower covariance case follows by the same method.

  Since the functional $E_P[(X-\mu_1)(Y-\mu_2)]$ is a linear function for $\mu_1$ and $P$, $M_X$ is closed and convex, by minimax theorem in Sion \cite{sion1958general}, we obtain
  \begin{align*}
    \overline{C}(X,Y)&=\max_{\mu_2\in M_Y}\min_{\mu_1\in M_X}\sup_{P\in \ccp}E_P[(X-\mu_1)(Y-\mu_2)]\\
    &=\max_{\mu_2\in M_Y}\sup_{P\in \ccp}\min_{\mu_1\in M_X}E_P[(X-\mu_1)(Y-\mu_2)]\\
    &=\sup_{P\in \ccp}\max_{\mu_2\in M_Y}\min_{\mu_1\in M_X}E_P[(X-\mu_1)(Y-\mu_2)]\\
    &=\sup_{P\in \ccp}C_P(X,Y).
  \end{align*}
\end{proof}
\begin{remark}
Theorem \ref{prop_rt} can be regarded as the definition of upper and lower covariances under multiple probability measures.
\end{remark}

We list the properties of upper and lower covariances.

\begin{proposition}\label{prop_property}
 For  random variables $X,Y$ and $Z$ on the sublinear expectation space $(\Omega,\mathcal{F},\pe)$ with $\pe[X^2]+\pe[Y^2]+\pe[Z^2]<\infty$, we have,
 \begin{itemize}
   \item [(1)] Symmetry: $\overline{C}(X,Y)=\overline{C}(Y,X), \ \ \ \underline{C}(X,Y)=\underline{C}(Y,X).$
   \item [(2)] Translation invariant: $\overline{C}(X+a,Y+b)=\overline{C}(X,Y), \ \ \ \underline{C}(X+a,Y+b)=\underline{C}(X,Y), \forall a,b\in\br$.
   \item[(3)] Sub-additivity: $\overline{C}(X+Y,Z)\leq\overline{C}(X,Z)+\overline{C}(Y,Z)$.
   \item[(4)] Sup-additivity: $\underline{C}(X+Y,Z)\geq\underline{C}(X,Z)+\underline{C}(Y,Z)$.
   \item [(5)] Positive homogeneity: $\forall a,b\in\br, ab\geq 0$,
   $$\overline{C}(aX,bY)=ab\overline{C}(X,Y),\ \ \ \underline{C}(aX,bY)=ab\underline{C}(X,Y).$$
   \item [(6)] Negative coupling: $\forall a,b\in\br, ab\le 0$,
   $$\overline{C}(aX,bY)=ab\underline{C}(X,Y).$$
   In particular,
   $$\underline{C}(X,Y)=-\overline{C}(-X,Y)=-\overline{C}(X,-Y).$$
\end{itemize}
\end{proposition}
\begin{proof}
  By Theorem \ref{prop_rt}, we immediately obtain (1) and (2).

  For (3), we have
\begin{align*}
    \overline{C}(X+Y,Z)&=\sup_{P\in\ccp}C_P(X+Y,Z)\\
    &=\sup_{P\in\ccp}(C_P(X,Z)+C_P(Y,Z))\leq\overline{C}(X,Z)+\overline{C}(Y,Z),
\end{align*}
(4) can be proved similarly.

  For (5), we have
    $$ \overline{C}(aX,bY)=\sup_{P\in\ccp}C_{P}(aX,bY)=\sup_{P\in\ccp}(abC_P(X,Y))=ab\overline{C}(X,Y),$$

   and for (6), we can obtain that
  $$ \overline{C}(aX,bY)=\sup_{P\in\ccp}C_{P}(aX,bY)=-ab\sup_{P\in\ccp}(-C_P(X,Y))=ab\underline{C}(X,Y).$$
\end{proof}

The following proposition provides the inequalities for variance and covariance.
\begin{proposition}\label{p38}
  Let $X$ and $Y$ be two random variables on the sublinear expectation space $(\Omega,\mathcal{F},\pe)$ with $\pe[X^2]+\pe[Y^2]<\infty$. Then
  \begin{itemize}
  \item [(1)] $\overline{V}(X+Y)\leq\overline{V}(X)+\overline{V}(Y)+2\overline{C}(X,Y).$
  \item [(2)] $\underline{V}(X+Y)\geq\underline{V}(X)+\underline{V}(Y)+2\underline{C}(X,Y).$
  \item [(3)] $\underline{V}\lt(\frac {X+Y}2\rt)-\overline{V}\lt(\frac {X-Y}2\rt)\le \underline{C}(X,Y)\le\overline{C}(X,Y)\le \overline{V}\lt(\frac{X+Y}{2}\rt)-\underline{V}\lt(\frac{X-Y}2\rt).$
  \item [(4)] $\lt|\overline{C}(X,Y)\rt|\le \sqrt{\overline{V}(X)\overline{V}(Y)}$.
  \end{itemize}
\end{proposition}
\begin{proof}
By Theorem \ref{vpt} and Theorem \ref{prop_rt}, we have
\begin{align*}
\overline{V}(X+Y)&=\sup_{P\in\ccp}V_P(X+Y)\\
&=\sup_{P\in\ccp}(V_P(X)+V_P(Y)+2C_P(X+Y))\\
&\leq\overline{V}(X)+\overline{V}(Y)+2\overline{C}(X,Y).
\end{align*}
Thus (1) holds. (2) is proved in the similar way.

For (3), it is clear that $\underline{C}(X,Y)\le\overline{C}(X,Y)$ and
\begin{align*}
\overline{C}(X,Y)&=\sup_{P\in\ccp}C_P(X,Y)\\
&=\sup_{P\in\ccp}\lt(V_P\lt(\frac{X+Y}{2}\rt)-V_P\lt(\frac{X-Y}{2}\rt)\rt)\\
&\leq\overline{V}\lt(\frac{X+Y}{2}\rt)-\underline{V}\lt(\frac{X-Y}{2}\rt).
\end{align*}
where the  identity $C_P(X,Y)=V_P\lt(\frac{X+Y}{2}\rt)-V_P\lt(\frac{X-Y}{2}\rt)$ has been utilized by Gnanadesikan and Ketternring \cite{GK}.

The proof for the rest inequality in (3) is similar.

  For (4),  we can obtain by
  $$ |\overline{C}(X,Y)|\leq\sup_{P\in\ccp}|C_{P}(X,Y)|\le \sup_{P\in\ccp}\sqrt{V_P(X)V_P(Y)}\le \sqrt{\overline{V}(X)\overline{V}(Y)}. $$
\end{proof}

\begin{remark}\label{r37}
If both $X$ and $Y$ have non-empty mean-uncertainty, i.e., $\pe[X]>-\pe[-X]$ and $\pe[Y]>-\pe[-Y]$, then we consider
$$\tilde{X}=\frac{X+\pe[-X]}{\Delta_X}, \ \ \ \tilde{Y}=\frac{Y+\pe[-Y]}{\Delta_Y},$$
where $\Delta_X=\pe[X]+\pe[-X]$ and $\Delta_Y=\pe[Y]+\pe[-Y]$.

It is easy to see that $M_{\tilde{X}}=M_{\tilde{Y}}=[0,1]$ and
$$\overline{C}(X,Y)=\frac{\overline{C}(\tilde{X},\tilde{Y})}{\Delta_X\Delta_Y}, \ \ \ \underline{C}(X,Y)=\frac{\underline{C}(\tilde{X},\tilde{Y})}{\Delta_X\Delta_Y}. $$
Thus we can approximately calculate  the upper covariance on the fixed interval $[0,1]\times[0,1]$ by exhaustive search on the domain $D_N=\{(\mu_i,\mu_j):\ \mu_i=\frac{i}{N},\ \mu_j=\frac{j}{N}, 0\leq i,j\leq N\}$ for large $N$.
\end{remark}

The following proposition shows that the bounded domains $M_X$ and $M_Y$ in the Definition \ref{def_cov} can be replaced by $\br$.

\begin{proposition}\label{pp310}
In Definition \ref{def_cov}, initially formulated within a bounded interval, can  extend to  the real space:
\begin{align}
\overline{C}(X,Y)&=\sup_{\mu_2\in \br}\inf_{\mu_1\in \br}\pe[(X-\mu_1)(Y-\mu_2)], \label{eq3-1}\\
\underline{C}(X,Y)&=\inf_{\mu_2\in \br}\sup_{\mu_1\in \br}\left(-\pe[-(X-\mu_1)(Y-\mu_2)]\right).\label{eq3-2}
\end{align}
\end{proposition}

\begin{proof}
Firstly, we have
\begin{align*}
\sup_{\mu_2\in \br}\inf_{\mu_1\in \br}\pe[(X-\mu_1)(Y-\mu_2)]&=\sup_{\mu_2\in \br}\inf_{\mu_1\in \br}\sup_{P\in\ccp}E_P[(X-\mu_1)(Y-\mu_2)]\\
&\geq\sup_{\mu_2\in \br}\sup_{P\in\ccp}\inf_{\mu_1\in \br}E_P[(X-\mu_1)(Y-\mu_2)]\\
&=\sup_{P\in\ccp}\sup_{\mu_2\in \br}\inf_{\mu_1\in \br}E_P[(X-\mu_1)(Y-\mu_2)]\\
&=\sup_{P\in\ccp}C_P(X,Y)=\overline{C}(X,Y).
\end{align*}
Secondly, it is clear that
$$\sup_{\mu_2\in \br}\inf_{\mu_1\in \br}\pe[(X-\mu_1)(Y-\mu_2)]\leq\sup_{\mu_2\in \br}\min_{\mu_1\in M_X}\pe[(X-\mu_1)(Y-\mu_2)].$$
We claim that
\begin{equation*}\label{eq3}
\sup_{\mu_2\in \br}\min_{\mu_1\in M_X}\pe[(X-\mu_1)(Y-\mu_2)]=\max_{\mu_2\in M_Y}\min_{\mu_1\in M_X}\pe[(X-\mu_1)(Y-\mu_2)]=\overline{C}(X,Y).
\end{equation*}
Indeed, let $h(\mu_2)=\min_{\mu_1\in M_X}\pe[(X-\mu_1)(Y-\mu_2)]$, then
\begin{align*}
h(\mu_2)&=\min_{\mu_1\in M_X}\sup_{P\in\ccp}E_P[(X-\mu_1)(Y-\mu_2)]\\
&=\sup_{P\in\ccp}\min_{\mu_1\in M_X}E_P[(X-\mu_1)(Y-\mu_2)]\\
&=\sup_{P\in\ccp}\min_{\mu_1\in M_X}\lt(C_P(X,Y)+(E_P[X]-\mu_1)(E_P[Y]-\mu_2)\rt).
\end{align*}
If $\mu_2>\ou_Y$ (resp. $\mu_2<\lu_Y$), then we take $\mu_1=-\pe[-X]$ (resp. $\mu_1=\pe[X]$) such that
$$(E_P[X]-\mu_1)(E_P[Y]-\mu_2)\leq 0,\ \ \forall P\in\ccp,$$
thus we obtain
$$h(\mu_2)\leq\sup_{P\in\ccp}C_P(X,Y)=\overline{C}(X,Y).$$
Finally, we prove that \cref{eq3-1} holds.

By (4) in Proposition \ref{prop_property}, $\underline{C}(X,Y)=-\overline{C}(-X,Y)$, thus \cref{eq3-2} holds.
\end{proof}

\begin{remark}\label{r31}

In fact, the operators $\sup$ and $\inf$ in Proposition \ref{pp310} can be replaced by $\max$ and $\min$ respectively.

 Furthermore, for any intervals $T_X\supseteq M_X$ and $T_Y\supseteq M_Y$, we have
$$\overline{C}(X,Y)=\max_{\mu_1\in T_X}\min_{\mu_2\in T_Y}\pe[(X-\mu_1)(Y-\mu_2)],$$
$$\underline{C}(X,Y)=\min_{\mu_1\in T_X}\max_{\mu_2\in T_Y}\lt(-\pe[-(X-\mu_1)(Y-\mu_2)]\rt).$$

We point out that, in Example \ref{em2}, we have
$$\min_{\mu_1\in\br}\max_{\mu_2\in\br} f(\mu_1,\mu_2)=+\infty.$$
{thus $\min\max$ operator is not so good to define upper covariance.}

\end{remark}

In our exploration of covariance and variance, we delve into the noteworthy implications arising from the equation $Y = X$. It's obvious that the upper (respectively, lower) covariance should align with the upper (respectively, lower) variance if $Y=X$ as in the classical probability theory. However, the direct relationship between these concepts isn't immediately apparent from Definition \ref{def_cov}. This gap in clarity underscores the value of considering an alternative definition of covariance, which we will introduce to not only clarify this relationship but also facilitate the computation of both upper and lower covariance, as will be further elaborated in Section 4.

The covariance $C_P(X,Y)$ defined in \cref{cov} can be written as the difference of two variances. Combined with \cref{eq1}, we have, $\forall \alpha,\beta\in\br$,
\begin{equation}\label{eq_cov_2}
    \begin{split}
       C_P(X,Y)=&V_P\lt(\frac{X+Y}{2}\rt)-V_P\lt(\frac{X-Y}{2}\rt)\\
               =&E_P\lt[\lt(\frac{X+Y}{2}-\alpha\rt)^2-\lt(\frac{X-Y}{2}-\beta\rt)^2\rt]\\
            &+\lt(E_P\lt[\frac{X+Y}{2}\rt]-\alpha\rt)^2-\lt(E_P\lt[\frac{X-Y}{2}\rt]-\beta\rt)^2.
    \end{split}
\end{equation}

For simplicity, we denote $U=\frac{X+Y}{2}$ and $V=\frac{X-Y}{2}$.
Thus we obtain
\begin{align*}
C_P(X,Y)&=\min_{\alpha\in\br}\max_{\beta\in\br}E_P[(U-\alpha)^2-(V-\beta)^2]\\
&=\max_{\beta\in\br}\min_{\alpha\in\br}E_P[(U-\alpha)^2-(V-\beta)^2].
\end{align*}

\begin{proposition}\label{pro1}
    Given two random variables $X$ and $Y$ on sublinear expectation space $(\Omega,\mathcal{F},\pe)$ with $\pe[X^2]+\pe[Y^2]<\infty$, we have
      \begin{align*}
        &\overline{C}(X,Y)=\max_{\beta\in \br}\min_{\alpha\in \br}\pe[(U-\alpha)^2-(V-\beta)^2],\\
        &\underline{C}(X,Y)=\min_{\alpha\in \br}\max_{\beta\in \br}\lt(-\pe[-(U-\alpha)^2+(V-\beta)^2]\rt),
      \end{align*}
      where $U=\frac{X+Y}{2}$ and $V=\frac{X-Y}{2}$.\par
In particular, if $Y=X$, then $\overline{C}(X,X)=\overline{V}(X)$ and $\underline{C}(X,X)=\underline{V}(X)$.

\end{proposition}

    This proposition can be regarded as another definition of upper and lower covariances, which can be easily proved by the following two lemmas.

\begin{lemma}\label{lem1}
        The domains of $\alpha$ and $\beta$ in Proposition \ref{pro1} can be replaced by $M_U$ and $M_V$ respectively. Indeed, we have
         $$\max_{\beta\in \br}\min_{\alpha\in \br}\pe[(U-\alpha)^2-(V-\beta)^2]=\max_{\beta\in M_{V}}\min_{\alpha\in M_{U}}\pe[(U-\alpha)^2-(V-\beta)^2],$$
         $$\min_{\alpha\in \br}\max_{\beta\in \br}\lt(-\pe[-(U-\alpha)^2+(V-\beta)^2]\rt)=\min_{\alpha\in M_{U}}\max_{\beta\in M_{V}}\lt(-\pe[-(U-\alpha)^2+(V-\beta)^2]\rt),$$
         where $U=\frac{X+Y}{2}$ and $V=\frac{X-Y}{2}$.
\end{lemma}
\begin{proof}
      We firstly consider the cases   for $\alpha$.

      If $\alpha>\ou_{U}:=\pe[U]$, we have
        \begin{align*}
          &\pe[(U-\alpha)^2-(V-\beta)^2]\\
          =&\pe[(U-\ou_U)^2-(V-\beta)^2+2(U-\ou_{U})(\ou_{U}-\alpha)]+(\ou_{U}-\alpha)^2\\
          \ge& \pe[(U-\ou_U)^2-(V-\beta)^2]-2(\pe[U]-\ou_{U})(\alpha-\ou_{U})+(\ou_{U}-\alpha)^2\\
          \ge& \pe[(U-\ou_U)^2-(V-\beta)^2].
        \end{align*}
        In the same way , if $\alpha<\lu_{U}:=-\pe[-U]$, we have
        $$ \pe[(U-\alpha)^2-(V-\beta)^2]\ge \pe[(U-\lu_{U})^2-(V-\beta)^2] ,$$
        which means that
        $$ \min_{\alpha\in \br}\pe[(U-\alpha)^2-(V-\beta)^2]=\min_{\alpha\in M_{U}}\pe[(U-\alpha)^2-(V-\beta)^2].$$
        And for any $\alpha\in M_{U}$ we can similarly obtain that,
        \begin{itemize}
          \item [(1)] if  $\beta>\ou_{V}:=\pe[V]$, $\pe[(U-\alpha)^2-(V-\beta)^2]\le \pe[(U-\alpha)^2-(V-\ou_{V})^2]$;
          \item [(2)] if $\beta<\lu_{V}:=-\pe[-V]$, $\pe[(U-\alpha)^2-(V-\beta)^2]\le \pe[(U-\alpha)^2-(V-\lu_{V})^2]$.
        \end{itemize}
        Therefore, the first equation holds. The proof for second equation is similar.
\end{proof}

\begin{lemma}
     Under the same assumptions of Proposition \ref{pro1}, we have
     \begin{itemize}
       \item [(i)]$\overline{C}(X,Y)=\max_{\beta\in \br}\min_{\alpha\in \br}\pe[(U-\alpha)^2-(V-\beta)^2].$\\
        Furthermore, there exists $\alpha^*\in M_U$ and $\beta^*\in M_V$ such that
        \[\overline{C}(X,Y)=\pe[(U-\alpha^*)^2-(V-\beta^*)^2].\]
       \item [(ii)]$\underline{C}(X,Y)=\min_{\alpha\in \br}\max_{\beta\in \br}\lt(-\pe[-(U-\alpha)^2+(V-\beta)^2]\rt).$\\
       Furthermore, there exists $\alpha^*\in M_U$ and $\beta^*\in M_V$ such that
        \[\underline{C}(X,Y)=-\pe[-(U-\alpha^*)^2+(V-\beta^*)^2].\]
        \end{itemize}
\end{lemma}

\begin{proof}
  Since the functional $E_P[(U-\alpha)^2-(V-\beta)^2]$ is a convex function for $\alpha$ and a linear function for $P$, $M_{U}$ is closed and convex, the minimax theorem can be used. By Theorem \ref{prop_rt}, Lemma \ref{lem1} and \cref{eq_cov_2}, we have
  \begin{align*}
    \overline{C}(X,Y)&=\sup_{P\in \ccp}C_P(X,Y)\\
    &=\sup_{P\in \ccp}\max_{\beta\in \br}\min_{\alpha\in \br}E_P[(U-\alpha)^2-(V-\beta)^2]\\
    &=\sup_{P\in \ccp}\max_{\beta\in M_{V}}\min_{\alpha\in M_{U}}E_P[(U-\alpha)^2-(V-\beta)^2]\\
    &=\max_{\beta\in M_{V}}\sup_{P\in \ccp}\min_{\alpha\in M_{U}}E_P[(U-\alpha)^2-(V-\beta)^2]\\
    &=\max_{\beta\in M_{V}}\min_{\alpha\in M_{U}}\sup_{P\in \ccp}E_P[(U-\alpha)^2-(V-\beta)^2]\\
    &=\max_{\beta\in \br}\min_{\alpha\in \br}\pe[(U-\alpha)^2-(V-\beta)^2].
  \end{align*}
  The existence of $(\alpha^*,\beta^*)$ is due to the continuity of $\pe[(U-\alpha)^2-(V-\beta)^2]$ in $(\alpha,\beta)$.

  The lower covariance case can be proved similarly.
\end{proof}

\begin{remark}
In Example \ref{emem3}, let $U=\frac{X+Y}{2}$ and $V=\frac{X-Y}{2}$, define function $h$ as
 \begin{equation*}
  h(\alpha,\beta)=\pe[(U-\alpha)^2-(V-\beta)^2]=
  \begin{cases}
    1+(\alpha+\frac{1}{2})^2-(\beta+\frac{1}{2})^2,&\alpha-\beta \ge 0\\
    1+(\alpha-\frac{1}{2})^2-(\beta-\frac{1}{2})^2,&\alpha-\beta<0
  \end{cases}.
\end{equation*}
Then we have
 $$ \overline{C}(X,Y)=\max_{\beta\in \br}\min_{\alpha\in \br} h(\alpha,\beta)=\max_{\beta\in M_V}\min_{\alpha\in M_U} h(\alpha,\beta)=1.$$
On the other hand,
$$ \min_{\alpha\in M_U}\max_{\beta\in M_V} h(\alpha,\beta)=\frac{7}{4}<+\infty=\min_{\alpha\in \br}\max_{\beta\in \br} h(\alpha,\beta).$$
which means the operators $\min\max$ and $\max\min$ in Proposition \ref{pro1} can not be interchanged.
\end{remark}

In the end of this section, we list following inequalities for upper and lower covariances, which generalize the corresponding inequalities for upper and lower variances in Appendix G of Walley \cite{walley1991statistical}.
\begin{proposition}\label{p1}
  Let $X$ and $Y$ be two random variables on sublinear expectation space $(\Omega,\mathcal{F},\pe)$ with $\pe[X^2]+\pe[Y^2]<\infty$. Let
  \begin{gather*}
    \rho_X=\frac{1}{2}(\overline{\mu}_X+\underline{\mu}_X),\  \rho_Y=\frac{1}{2}(\overline{\mu}_Y+\underline{\mu}_Y),\\
    \Delta_X=\overline{\mu}_X-\underline{\mu}_X,\  \Delta_Y=\overline{\mu}_Y-\underline{\mu}_Y.
  \end{gather*}
We have the following results:
\begin{itemize}
  \item [(1)] When $\mu_1=\overline{\mu}_X,\mu_2=\overline{\mu}_Y$ or $\mu_1=\underline{\mu}_X,\mu_2=\underline{\mu}_Y$, we have
  \begin{gather*}
    \overline{C}(X,Y)\le \pe[(X-\mu_1)(Y-\mu_2)]\le \overline{C}(X,Y)+\Delta_X\Delta_Y,\\
    \underline{C}(X,Y)\le -\pe[-(X-\mu_1)(Y-\mu_2)]\le \underline{C}(X,Y)+\Delta_X\Delta_Y.
  \end{gather*}
  When $\mu_1=\underline{\mu}_X,\mu_2=\overline{\mu}_Y$ or $\mu_1=\overline{\mu}_X,\mu_2=\underline{\mu}_Y$, we have
  \begin{gather*}
    \overline{C}(X,Y)-\Delta_X\Delta_Y\le \pe[(X-\mu_1)(Y-\mu_2)]\le \overline{C}(X,Y)+\Delta_X\Delta_Y,\\
    \underline{C}(X,Y)-\Delta_X\Delta_Y\le -\pe[-(X-\mu_1)(Y-\mu_2)]\le \underline{C}(X,Y)+\Delta_X\Delta_Y.
  \end{gather*}
  \item [(2)] Let
  \begin{gather*}
    \overline{M}(X,Y)=\max\{\underline{\mu}_X\underline{\mu}_Y,\overline{\mu}_X\underline{\mu}_Y,\underline{\mu}_X\overline{\mu}_Y,\overline{\mu}_X\overline{\mu}_Y\},\\
    \underline{M}(X,Y)=\min\{\underline{\mu}_X\underline{\mu}_Y,\overline{\mu}_X\underline{\mu}_Y,\underline{\mu}_X\overline{\mu}_Y,\overline{\mu}_X\overline{\mu}_Y\}.
  \end{gather*}
Then we have
    \begin{gather*}
    \overline{C}(X,Y)+\underline{M}(X,Y)\le\pe[XY]\le \overline{C}(X,Y)+\overline{M}(X,Y),\\
    \underline{C}(X,Y)+\underline{M}(X,Y)\le-\pe[-XY]\le \underline{C}(X,Y)+\overline{M}(X,Y).
  \end{gather*}
  \item [(3)] The bound of upper and lower covariances:
  \begin{gather*}
    \pe[(X-\rho_X)(Y-\rho_Y)]-\frac 14 \Delta_X\Delta_Y\le\overline{C}(X,Y)\le\pe[(X-\rho_X)(Y-\rho_Y)]+\frac 14 \Delta_X\Delta_Y,\\
    -\pe[-(X-\rho_X)(Y-\rho_Y)]-\frac 14 \Delta_X\Delta_Y\le\underline{C}(X,Y)\le-\pe[-(X-\rho_X)(Y-\rho_Y)]+\frac 14 \Delta_X\Delta_Y.
  \end{gather*}
  \item [(4)] The bound of the gap between upper and lower covariance:
  \[ 0\le\overline{C}(X,Y)-\underline{C}(X,Y)\le\pe[(X-\rho_X)(Y-\rho_Y)]+\pe[-(X-\rho_X)(Y-\rho_Y)]+\frac 12 \Delta_X\Delta_Y.
  \]
\end{itemize}
\end{proposition}
\begin{proof}
  (1) For any real numbers $\mu_1,\mu_2\in\mathbb{R}$ and $P\in\mathcal{P}$, we have
  \[ E_P[(X-\mu_1)(Y-\mu_2)]=C_P(X,Y)+(E_P[X]-\mu_1)(E_P[Y]-\mu_2).
    \]
   By Theorem \ref{prop_rt}, we have
   \begin{equation}
   \begin{aligned}\label{eq_1}
    &\pe[(X-\mu_1)(Y-\mu_2)]\\
    =&\sup_{P\in\mathcal{P}}E_P[(X-\mu_1)(Y-\mu_2)]\\
    =&\sup_{P\in\mathcal{P}}(C_P(X,Y)+(E_P[X]-\mu_1)(E_P[Y]-\mu_2))\\
    \le&\overline{C}(X,Y)+\sup_{P\in\mathcal{P}}\left((E_P[X]-\mu_1)(E_P[Y]-\mu_2)\right),\\
   \end{aligned}
  \end{equation}
and
\begin{equation}
  \begin{aligned}\label{eq_2}
    &\overline{C}(X,Y)\\
    =&\sup_{P\in\ccp}C_P(X,Y)\\
    =&\sup_{P\in\ccp}(E_P[(X-\mu_1)(Y-\mu_2)]-(E_P[X]-\mu_1)(E_P[Y]-\mu_2))\\
    \le&\pe[(X-\mu_1)(Y-\mu_2)]-\inf_{P\in\ccp}\left((E_P[X]-\mu_1)(E_P[Y]-\mu_2)\right).\\
\end{aligned}
\end{equation}

If $\mu_1=\overline{\mu}_X,\mu_2=\overline{\mu}_Y$ or $\mu_1=\underline{\mu}_X,\mu_2=\underline{\mu}_Y$, then we have
$$0\leq \inf_{P\in\ccp}(E_P[X]-\mu_1)(E_P[Y-\mu_2])\leq \sup_{P\in\cp}(E_P[X]-\mu_1)(E_P[Y-\mu_2])\leq \Delta_X\Delta_Y.$$

The first inequality in (1) holds. The other inequalities in (1) can be proved similarly.

(2) By  Theorem \ref{prop_rt} and the fact that
\[ E_P[XY]=C_P(X,Y)+E_P[X]E_P[Y],
  \]
  we can get the result in the similar way as in (1).

  (3) Let $\mu_1=\rho_X$ and $\mu_2=\rho_Y$, then
$$-\frac{1}{4}\Delta_X\Delta_Y\leq(E_P[X]-\rho_X)(E_P[Y]-\rho_Y)\leq\frac{1}{4}\Delta_X\Delta_Y, \ \ \forall P\in\ccp.$$
  According to \cref{eq_1} and \cref{eq_2},we get
  \begin{gather*}
        \pe[(X-\rho_X)(Y-\rho_Y)]-\frac 14 \Delta_X\Delta_Y\le\overline{C}(X,Y)\le\pe[(X-\rho_X)(Y-\rho_Y)]+\frac 14 \Delta_X\Delta_Y.
  \end{gather*}
Similarly, we can obtain the case of lower covariance.

  (4) It is obvious according to (3).
\end{proof}

By (3) of Proposition \ref{p1}, we immediately have the following corollary.
\begin{corollary}\label{c316}
Let $X$ and $Y$ be two random variables on the sublinear expectation space $(\Omega,\mathcal{F},\pe)$ with $\pe[X^2]+\pe[Y^2]<\infty$. Furthermore, if $\ou_X=\lu_X=\mu_X$, then
\begin{align*}
\overline{C}(X,Y)&=\pe[(X-\mu_X)(Y-\rho_Y)],\\
\underline{C}(X,Y)&=-\pe[-(X-\mu_X)(Y-\rho_Y)].
\end{align*}
\end{corollary}

\begin{remark}
If both of $X$ and $Y$ with zero means have variance-certainty, i.e.,
$$\pe[X]=\pe[-X]=0,\ \ \ \pe[Y]=\pe[-Y]=0,$$
$$\pe[X^2]=-\pe[-X^2]=:V(X), \ \ \ \pe[Y^2]=-\pe[-Y^2]=:V(Y),$$
then upper and lower correlation coefficients $\overline{\rho}_{XY}$ and $\underline{\rho}_{XY}$ can be calculated as
$$\overline{\rho}_{XY}=\frac{\overline{C}(X,Y)}{\sqrt{V(X)V(Y)}}=\frac{\pe[XY]}{\sqrt{V(X)V(Y)}}, $$
$$\underline{\rho}_{XY}=\frac{\underline{C}(X,Y)}{\sqrt{V(X)V(Y)}}=\frac{-\pe[-XY]}{\sqrt{V(X)V(Y)}}.$$
The related portfolio optimization problem with such ambiguous correlation was solved in Fouque et al. \cite{FPW}.
\end{remark}

\section{Calculation of upper and lower covariance}
While calculating lower variance is often straightforward due to the optimal probability measure typically being an extreme point of $\cp$ (see (2) in Theorem \ref{vpt}), computing upper variance and upper covariance proves more challenging. This is because the optimal probability measure for these cases often doesn't reside at an extreme point (Theorem \ref{vpt}, Example \ref{em1}, Example \ref{emem3}). However, Proposition \ref{prop_property}  establishes that $\underline{C}(X,Y) = -\overline{C}(-X,Y)$, enabling us to focus solely on upper covariance calculations.

In this section, we consider the case of $\cp$ being generalized by the finitely many probability measures, i.e., $\cp=\{P_1,\cdots, P_K\}$.

In this case, let $X$ and $Y$ be two random variables on the sublinear expectation space $(\Omega,\mathcal{F},\pe)$ with $\pe[X^2]+\pe[Y^2]<\infty$. We calculate upper covariance by Proposition \ref{pro1} based on a simple algorithm of calculation for the upper variance introduced by Li et al. \cite{li2023}.
In this case, the upper and lower covariance is defined by
    \begin{gather}\label{eee1}
    \overline{C}(X,Y)=\max_{\beta\in\br}\min_{\alpha\in\br}\max_{1\leq i\leq K}\{E_{P_i}[(U-\alpha)^2-(V-\beta)^2]\},
\end{gather}
where $U=\frac{X+Y}{2}$ and $V=\frac{X-Y}{2}$.

We  rewrite \cref{eee1} as
\begin{equation}\label{eq_ucov_rewrite}
\overline{C}(X,Y)=\max_{\beta\in\br}\min_{\alpha\in\br}\max_{1\leq i\leq K}\{\alpha^2-2\mu_i\alpha-\beta^2+2\nu_i\beta+\kappa_i\},
\end{equation}
where $\mu_i=\frac{E_{P_i}[X]+E_{P_i}[Y]}{2}$, $\nu_i=\frac{E_{P_i}[X]-E_{P_i}[Y]}{2}$ and $\kappa_i=E_{P_i}[XY]$.

The following lemma is proved in Li et al. \cite{li2023}.
\begin{lemma}\label{lem_minmax}
Let
$$V=\min_{\alpha\in\br}\max_{1\leq i\leq K}\{\alpha^2-2\mu_i\alpha+d_i\},$$
then
$$V=\max\left\{\max_{1\leq i\leq K}(d_i-\mu_i^2),\max_{1\leq i<j\leq K}h_{ij}(\mu_{ij})\right\},$$
where $h_{ij}(x)=x^2-2\mu_ix+d_i$, \begin{align*}
 \mu_{ij}=
 \begin{cases}
    \lt(\underline{\mu}_{ij}\vee\frac{d_{j}-d_{i}}{2(\mu_{j}-\mu_{i})}\rt)\wedge\overline{\mu}_{ij}, & \mu_i\neq\mu_j, \\
    \mu_i, & \mu_i=\mu_j,
 \end{cases}  \ \ \ \ \ 1\leq i<j\leq K,
\end{align*}
and $\overline{\mu}_{ij}=\mu_i\vee\mu_j, \ \underline{\mu}_{ij}=\mu_i\wedge\mu_j$.
\end{lemma}

Similar to Proposition \ref{pp25}, the upper covariance under $K$-probabilities can be represented by the maximum of upper covariances under $2$-probabilities.

\begin{proposition}\label{lem_ucov}
  The upper covariance can be calculated by
  $$ \overline{C}(X,Y)=\max_{1\le i<j\le K} \overline{C}_{ij}(X,Y),$$
  where $\overline{C}_{ij}(X,Y)$ is upper covariance of $X$ and $Y$ under two probability measures $P_i$ and $P_j$.
\end{proposition}

%
\begin{proof}
  We consider the upper covariance $\overline{C}(X,Y)$ as defined in \cref{eq_ucov_rewrite}. Without loss of generality, we assume $\mu_1\leq\mu_2\leq\cdots\leq\mu_K$.
  \par For $1\le i<j\le K$, we define $\alpha_{ij}(\beta)$ as
  \begin{equation*}
      \alpha_{ij}(\beta)=
      \begin{cases}
          (\mu_i \vee (d_{ij}\beta+e_{ij}))\wedge \mu_j,& \mu_i<\mu_j,\\
          \mu_i, &\mu_i=\mu_j,
      \end{cases}
  \end{equation*}
  where $d_{ij}=\tfrac{\nu_j-\nu_i}{\mu_j-\mu_i}$ and $e_{ij}=\tfrac{\kappa_j-\kappa_i}{2(\mu_j-\mu_i)}$.
  \par Fix $\beta\in \br$, according to Lemma \ref{lem_minmax}, we have
  $$\min_{\alpha \in \br}\max_{1\le i\le K}\{ \alpha^2-2\mu_i\alpha-\beta^2+2\nu_i\beta+\kappa_i \}=\max\lt\{\max_{1\le i<j\le K}h_{ij}^{\beta}(\alpha_{ij}(\beta)),\max_{1\le i\le K}\lt\{ -\beta^2+2\nu_i\beta+\kappa_i-\mu_i^2 \rt\}\rt\},$$
  where
  \begin{equation*}
      h_{ij}^{\beta}(x)=x^2-2\mu_ix-\beta^2+2\nu_i\beta+\kappa_i.
  \end{equation*}
  Therefore,
  \begin{align*}
      \overline{C}(X,Y)=& \max_{\beta\in \br}\max\lt\{\max_{1\le i<j\le K}h_{ij}^{\beta}(\alpha_{ij}(\beta)),\max_{1\le i\le K}\lt\{ -\beta^2+2\nu_i\beta+\kappa_i-\mu_i^2\rt\}\rt\}\\
      =& \max\lt\{\max_{\beta\in \br}\max_{1\le i<j\le K}h_{ij}^{\beta}(\alpha_{ij}(\beta)),\max_{\beta\in \br}\max_{1\le i\le K}\lt\{ -\beta^2+2\nu_i\beta+\kappa_i-\mu_i^2 \rt\}\rt\}\\
      =& \max\lt\{\max_{1\le i<j\le K}\max_{\beta\in \br}h_{ij}^{\beta}(\alpha_{ij}(\beta)),\max_{1\le i\le K}C_{P_i}(X,Y)\rt\}\\
      =&\max_{1\le i<j\le K}\max\lt\{ \max_{\beta\in \br}h_{ij}^{\beta}(\alpha_{ij}(\beta)), C_{P_i}(X,Y), C_{P_j}(X,Y) \rt\}.
  \end{align*}
  If we consider upper covariance under $\{P_i, P_j\}$, then we can easily get
  $$ \overline{C}_{ij}(X,Y)=\max\lt\{ \max_{\beta\in \br}h_{ij}^{\beta}(\alpha_{ij}(\beta)), C_{P_i}(X,Y), C_{P_j}(X,Y) \rt\}. $$
  Therefore, we obtain
  $$\overline{C}(X,Y)=\max_{1\le i<j\le K} \overline{C}_{ij}(X,Y).$$
\end{proof}
\begin{corollary}
The upper variance of $X$ under $\cp=\{P_1,\cdots,P_k\}$ can be calculated by
$$\overline{V}(X)=\max_{1\leq i<j\leq K}\overline{V}_{ij}(X),$$
where $\overline{V}_{ij}(X)$ is the upper variance of $X$ under two probability measures $P_i$ and $P_j$.
\end{corollary}

Compared with Definition \ref{def_cov} and Proposition \ref{pro1}, it will be more easily to calculate upper covariance under two probability measures via Definition \ref{def_cov}.

\begin{theorem}\label{thm_cal_cov} Given $\cp=\{P_1,\cdots, P_K\}$, we denote $a_i, b_i$ and $c_i$ by $E_{P_i}[X],E_{P_i}[Y]$ and $ C_{P_i}(X,Y)$, $1\leq i\leq K$, respectively. The upper covariance of $X$ and $Y$ under $\cp$ can be calculated by
  $$ \overline{C}(X,Y)=\max\lt\{ \max_{1\le i\le K}C_{P_i}(X,Y), \max_{1\leq i\leq j\leq K}q_{ij}(\tilde{\mu}_{ij})\rt \}, $$
  where   $q_{ij}(x)$ is defined by
    \begin{equation}\label{eq_qij}
   q_{ij}(x)=\left\{ \begin{aligned}
    &\frac{1}{(b_j-b_i)}\lt( (x-b_i)(x-b_j)(a_i-a_j)+(x-b_i)(c_j-c_i)\rt)+c_i,\  & a_i\neq a_j \ \text{and}\ b_i\neq b_j\\
    & c_i, \ \ \ & \text{otherwise}
    \end{aligned}\right.
  \end{equation}
 and $\tilde{\mu}_{ij}$ is defined by
  \begin{equation}\label{eq_muij}
    \tilde{\mu}_{ij}=\lt\{ \begin{aligned}&\lt(\lt(\frac{c_j-c_i}{2(a_j-a_i)}+\frac{b_i+b_j}{2}\rt)\vee \underline{b}_{ij}\rt)\wedge \overline{b}_{ij}, \ \ & a_i\neq a_j \ \text{and} \ b_i\neq b_j, \\
    &b_i, &\text{otherwise}
     \end{aligned}
     \rt.
  \end{equation}
where $\underline{b}_{ij}=b_i\wedge b_j$ and $\overline{b}_{ij}=b_i\vee b_j$.
\end{theorem}
\begin{proof}  We only need to calculate $\overline{C}_{ij}(X,Y) $ for $1\le i<j\le K$.
  \par Without loss of generality, we assume that $b_i\le b_j$. Therefore, $\mu_2\in[b_i,b_j].$
Define function $f_{ij}(\mu_1,\mu_2)$ as
\begin{align*}
  f_{ij}(\mu_1,\mu_2):=&\max\{ c_i+(\mu_1-a_i)(\mu_2-b_i),c_j+(\mu_1-a_j)(\mu_2-b_j) \}\\
  &=
  \begin{cases}
    c_i+(\mu_1-a_i)(\mu_2-b_i), & g_{ij}(\mu_1,\mu_2)\ge 0, \\
    c_j+(\mu_1-a_j)(\mu_2-b_j), & g_{ij}(\mu_1,\mu_2) <0,
  \end{cases}
 \end{align*}
 where $g_{ij}(\mu_1,\mu_2):=(b_j-b_i)\mu_1+(a_j-a_i)\mu_2+a_ib_i-a_jb_j-c_j+c_i$.
 \par If $a_i=a_j$ or $b_i=b_j$, we have $f_{ij}(\mu_1,\mu_2)=\max\{ c_i,c_j \}.$ Next, we assume that $a_i\neq a_j$ and $b_i<b_j$.
 \par (1) If $g_{ij}\ge 0$, which means $\mu_1\ge \mu_{ij}^*(\mu_2)$, where
  $$\mu_{ij}^*(\mu_2):=\frac{1}{(b_j-b_i)} (c_j-c_i+a_jb_j-a_ib_i-(a_j-a_i)\mu_2).$$
  Since $\mu_2-b_i\ge 0$, in order to select $\mu_1$ to reach the minimum value, we have $\mu_1=\mu_{ij}^*(\mu_2)$. Take it into $f_{ij}$ we obtain
  \begin{equation*}
    f_{ij}(\mu_{ij}^*(\mu_2),\mu_2)=q_{ij}(\mu_2),
  \end{equation*}
  where $q_{ij}$ is defined in \eqref{eq_qij}.\\
  The maximum value will be reached at $b_i,b_j$ or $\tilde{\mu}_{ij}$, where $\tilde{\mu}_{ij}$ is defined in \eqref{eq_muij}.\\
  If $\mu_2=b_i$ or $b_j$, we have $q_{ij}(\mu_2)=c_i$ or $c_j$.
  \par (2) The case $g_{ij}<0$ is similar to analyse and we can obtain the same result. \\
  In conclusion, we have
  $$ \overline{C}_{ij}(X,Y)=\max\{ c_i,c_j , q_{ij}(\tilde{\mu}_{ij})\},$$
  which can directly obtain the result of the theorem.\end{proof}

Now we can provide our algorithm as following:\\
\textbf{Algorithm of Upper Covariance}
\begin{itemize}
  \item [Step 1:] Calculate $a_i=E_{P_i}[X], b_i=E_{P_i}[Y],c_i=C_{P_i}(X,Y), 1\le i\le K$.
  \item [Step 2:]For $1\le i<j\le K$, calculate $q_{ij}(\tilde{\mu}_{ij})$ and $\tilde{\mu}_{ij}$ defined in \cref{eq_qij} and \cref{eq_muij}.
  \item [Step 3:] Output the upper covariance $$\overline{C}(X,Y)=\max\lt\{ \max_{1\le i\le K}c_i, \max_{1\le i<j\le K}q_{ij}(\tilde{\mu}_{ij})\rt\}.$$
\end{itemize}

\begin{remark}
  According to (4) in Proposition \ref{prop_property} , we can easily get
$$ \underline{C}(X,Y)=-\overline{C}(X,-Y),$$
from which we can use the algorithm of upper covariance to calculate the lower covariance.
\end{remark}

In practice, the parameters $(a_i,b_i,c_i), 1\leq i\leq K$ in Theorem \ref{thm_cal_cov} can be estimated from data.
\par For example, let $(X,Y)$ be the daily returns of two stocks. In the real market, we can obtain the daily return data of these two stocks $\{ x_i\}_{i\in I},\{ y_i\}_{i\in I}$ (resp. $\{x_j\}_{j\in J},\{ y_j\}_{j\in J}$) from bull (resp. bear) market, where $I$ (resp. $J$) denotes the periods of bull (resp. bear) market. Then $K=2$ and we can estimate the sample means as
$$ \hat{a}_1=\frac{\Sigma_{i\in I}x_i}{\vert I\vert},\quad \hat{a}_2=\frac{\Sigma_{j\in J}x_j}{\vert J\vert},\quad \hat{b}_1=\frac{\Sigma_{i\in I}y_i}{\vert I\vert},\quad \hat{b}_2=\frac{\Sigma_{j\in J}y_j}{\vert J\vert} $$
and sample covariances as
$$\hat{c}_{1}=\frac{1}{|I|-1}\sum_{i\in I}(x_i-\hat{a}_1)(y_i-\hat{b}_1),\quad \hat{c}_{2}=\frac{1}{|J|-1}\sum_{j\in J}(x_j-\hat{a}_2)(y_j-\hat{b}_2).$$
Then we can take $a_i=\hat{a}_i,\ b_i=\hat{b}_i,\ c_i=\hat{c}_i,\ i=1,2$ to calculate the upper and lower covariances by our algorithm.

In the end of this section, it is natural to consider the notion of upper and lower covariance matrices under multiple probability measures, which can be widely used in many fields, especially in robust portfolio selection models.

Let ${X}=(X_1,\cdots,X_n)^T$ be random vector on sublinear expectation space $(\Omega,\cf,\pe)$ with finite second moments. Formally, we define the upper and lower covariance matrices as
\begin{equation}\label{co1}\overline{\Sigma}=\left(\overline{C}(X_i,X_j)\right)_{1\leq i,j\leq n}, \ \ \ \underline{\Sigma}=\left(\underline{C}(X_i,X_j)\right)_{1\leq i,j\leq n}.
\end{equation}
They are symmetric but not positive semi-definite in general, see Example \ref{em54}. We also point out that the upper covariance matrix for two random variables is still positive semi-definite by (4) in Proposition \ref{p38}.

\begin{example}\label{em54}
  Let $X=(X_1,X_2,X_3)$ be trivariate normally distributed under $P_1:=N(\mu_1,\Sigma_1)$ and $P_2:=N(\mu_2,\Sigma_2)$, where $\mu_i$ and $\Sigma_i$, $i=1,2$ are the mean vector and covariance matrix respectively defined as:
  $$ \mu_1=(-1,1,0)^T,\quad \mu_2=(-2,1,-1)^T,$$
  \begin{equation*}
    \Sigma_1=
      \begin{pmatrix}
          2.00&-1.20&-1.98 \\
      -1.20&2.00&2.55\\
      -1.98&2.55&4.00
      \end{pmatrix},\
      \Sigma_2=
      \begin{pmatrix}
          2.00&0.40&2.83 \\
        0.40&2.00&-1.98\\
        2.83&-1.98&4.00
      \end{pmatrix}.
  \end{equation*}

  Then
  \begin{equation*}
    \overline{\Sigma}=
      \begin{pmatrix}
          2.25&0.40&2.83 \\
        0.40&2.00&2.55\\
        2.83&2.55&4.25
      \end{pmatrix},\
      \underline{\Sigma}=
      \begin{pmatrix}
          2.00&-3.46&-2.42 \\
        -3.46&2.00&-2.40\\
        -2.42&-2.40&4.00
      \end{pmatrix}.
  \end{equation*}
  It can be verified that $\overline{\Sigma}$ and $\underline{\Sigma}$ are not positive semi-definite.
\end{example}

It is still open for us to define the upper and lower covariance matrices under multiple probability measures which are positive semi-definite. In particular, if both of $\overline{\Sigma}$ and $\underline{\Sigma}$ defined by \cref{co1} are positive semi-definite, they can be used to characterize box uncertainty on the covariance matrix in the robust portfolio selection model (see T\"{u}t\"{u}nc\"{u} and Koenig \cite{TK}, Fabozzi et al. \cite{FHZ}).

Instead of using upper and lower covariance matrices based on multiple probabilities, we consider the set of covariance matrices introduced by the convex hull of multiple probabilities. Similar to Theorem \ref{prop_rt}, we have following definition.

\begin{definition}
Let $X=(X_1,\cdots,X_n)^T$ be an $n$-dimensional random vector on the sublinear expectation space $(\Omega,\cf,\pe)$ with $\pe[||X||^2]<\infty$, where $||X||$ is the Euclidean norm of $X$. The uncertainty set of covariance matrices, denoted by $\ccs$, is defined by
$$\ccs=\lt\{\Sigma:=E_P[(X-E_P[X])(X-E_P[X])^T], \ \ \forall P\in\ccp\rt\}.$$
\end{definition}
It is easily seen that all the elements in $\ccs$ are positive semi-definite. Furthermore, we have the following inequalities.

\begin{proposition}
Let $X=(X_1,\cdots,X_n)^T$ be an $n$-dimensional random vector on the sublinear expectation space $(\Omega,\cf,\pe)$ with $\pe[||X||^2]<\infty$. For each $\Sigma=(c_{ij})_{1\leq i,j\leq n}\in\ccs$, we have
$$\underline{C}(X_i,X_j)\leq c_{ij}\leq\overline{C}(X_i,X_j),\ \ \ \ \forall\ 1\leq i,j\leq n.$$
\end{proposition}


\section{Application: Quadratic programming problem}

Thanks to Theorem \ref{prop_rt}, the notion of upper covariance can be used to solve a class of quadratic programming problem.

\begin{proposition}
Let $X$ and $Y$ be two random variable under $\cp=\{P_1,\cdots,P_K\}$ satisfying $\pe[X^2]+\pe[Y^2]<+\infty$. We denote $\bbm=(E_{P_1}[X],\cdots,E_{P_K}[X])^T\in\br^K$,
$\bbn=(E_{P_1}[Y],\cdots,E_{P_K}[Y])^T\in\br^K$ and $\bbk=(E_{P_1}[XY],\cdots,E_{P_K}[XY])^T\in\br^K$. Then we have
\begin{equation}\label{e51}
\overline{C}(X,Y)=\max_{\bbl\in\Delta^K}(\bbl^T\bbk-\bbl^T\bbm\bbn^T\bbl),
\end{equation}
where $\Delta^K=\{\bbl=(\lambda_1,\cdots,\lambda_K)\in\br^K:\ \ \sum_{i=1}^K\lambda_i=1, \ \ \lambda_i\geq 0,\ 1\leq i\leq K\}.$
\end{proposition}
\begin{proof}
By Theorem \ref{prop_rt}, we have
$$\overline{C}(X,Y)=\sup_{P\in\ccp}C_P(X,Y),$$
where $\ccp=\{P_{\bbl}: P_{\bbl}=\lambda_1P_1+\cdots+\lambda_KP_K, \forall\ \bbl=(\lambda_1,\cdots,\lambda_K)^T\in\Delta^K\}$.

Thus we obtain
$$\overline{C}(X,Y)=\max_{\bbl\in\Delta^K}C_{P_{\bbl}}(X,Y),$$
where
\begin{align*}
C_{P_{\bbl}}(X,Y)=E_{P_{\bbl}}[XY]-E_{P_{\bbl}}[X]E_{P_{\bbl}}[Y]=\bbl^T\bbk-\bbl^T\bbm\bbn^T\bbl.
\end{align*}
\end{proof}

We note that the optimization problem (\ref{e51}) can be written as the quadratic programming
$$\max_{\bbl\in\Delta^K}(\bbl^T\bbk-\bbl^TQ\bbl),$$
where $Q=\frac{1}{2}(\bbm\bbn^T+\bbn\bbm^T)$. Unfortunately, $Q$ may be an indefinite matrix, see Example \ref{em51}.
\begin{example}\label{em51}
We take $\bbm=(1,1)^T, \bbn=(1,0)^T$, then
$$Q=\left(
      \begin{array}{cc}
        1 & \frac{1}{2} \\
        \frac{1}{2} & 0 \\
      \end{array}
    \right),
$$
which is indefinite.
\end{example}

There are many numerical methods to obtain the approximate solution. We can obtain the solution by the algorithm for the calculation of upper covariance.

\begin{theorem} \label{th1}
Given $\bbm=(\mu_1,\cdots,\mu_K)^T\in\br^K$,
$\bbn=(\nu_1,\cdots,\nu_K)^T\in\br^K$ and $\bbk=(\kappa_1,\cdots,\kappa_K)^T\in\br^K$. the exact solution of the following quadratic programming:
$$V=\max_{\bbl\in\Delta^K}(\bbl^T\bbk-\bbl^T\bbm\bbn^T\bbl)$$
is given by
$$V=\max\lt\{ \max_{1\le i\le K}\{ \kappa_i-\mu_i\nu_i \}, \max_{1\leq i\leq j\leq K}q_{ij}(\tilde{\mu}_{ij})\rt \},$$
where
  $q_{ij}(x)$ is defined by
    \begin{equation}\label{eq_qij2}
      q_{ij}(x)=\lt\{\begin{aligned}
          \frac{1}{(\nu_j-\nu_i)}\lt( (x-\nu_i)(x-\nu_j)(\mu_i-\mu_j)+(x-\nu_i)(\kappa_j-\kappa_i-\mu_j\nu_j+\mu_i\nu_i)\rt)+\kappa_i-\mu_i\nu_i,\\
          \ \mu_i\neq\mu_j\ \text{and}\ \nu_i\neq\nu_j \\
          \kappa_i-\mu_i\nu_i, \ \ \ \ \ \ \ \ \ \ \ \ \ \ \ \ \ \ \ \ \ \ \ \ \ \ \ \ \ \ \ \ \  \ \ \ \ \ \ \ \ \ \text{otherwise}
      \end{aligned}
      \rt.
  \end{equation}
 and $\tilde{\mu}_{ij}$ is defined by
  \begin{equation}\label{eq_muij2}
    \tilde{\mu}_{ij}=\lt\{
    \begin{aligned}
        \lt(\lt(\frac{\kappa_j-\mu_j\nu_j-\kappa_i+\mu_i\nu_i}{2(\mu_j-\mu_i)}+\frac{\nu_i+\nu_j}{2}\rt)\vee \underline{\nu}_{ij}\rt)\wedge \overline{\nu}_{ij}, \ \ \mu_i\neq\mu_j\ \text{and}\ \nu_i\neq\nu_j \\
        \nu_i, \ \ \ \ \ \ \ \ \ \ \ \ \ \ \ \ \ \ \ \ \ \ \ \ \ \ \ \ \ \ \ \ \ \ \ \ \ \ \ \ \ \ \ \ \ \ \ \ \text{otherwise}
    \end{aligned}\rt.
    \end{equation}
    where $\underline{\nu}_{ij}=\min\{\nu_i,\nu_j\}, \overline{\nu}_{ij}=\max\{\nu_i, \nu_j\}$.

If there exists $i_0$ such that $V=\kappa_{i_0}-\mu_{i_0}\nu_{i_0}$, then the optimal $\bbl^*$ is given by $\lambda_{i_0}^*=1$ and $\lambda_{j}^*=0$, $j\neq i_0$.
Otherwise, there exists $1\leq i_0<j_0\leq K$ such that $V=q_{i_0j_0}(\tilde{\mu}_{i_0j_0}) $, then the optimal $\bbl^*$ is given by $\lambda_{i_0}^*=\frac 12-\frac{\kappa_{j_0}-\mu_{j_0}\nu_{j_0}-\kappa_{i_0}+\mu_{i_0}\nu_{i_0}}{2(\mu_{j_0}-\mu_{i_0})(\nu_{j_0}-\nu_{i_0})}$ and $\lambda_{j_0}^*=1-\lambda_{i_0}^*$, $\lambda_j^*=0$, $j\neq i_0,j_0$.

\end{theorem}

\begin{proof} Actually, $V$ can be regarded as the upper covariance of two random variable $X$ and $Y$ under the set of probability measures $\{P_1,\cdots, P_K\}$ with $E_{P_i}[X]=\mu_i$, $E_{P_i}[Y]=\nu_i$ and $E_{P_i}[XY]=\kappa_i$, $1\leq i\leq K$.
 \par According to Proposition \ref{lem_ucov}, we have the following two cases:\\
(1) There exists $i_0$ such that $V=C_{P_{i_0}}(X,Y)=\kappa_{i_0}-\mu_{i_0}\nu_{i_0}$, then the optimal $\bbl^*$ is given by $\lambda_{i_0}^*=1$ and $\lambda_{j}^*=0$, $j\neq i_0$.\\
(2) There exists $1\leq i_0<j_0\leq K$ such that $V=\overline{C}_{i_0j_0}(X,Y)$. We consider the following maximize problem:
$$ \overline{C}_{i_0j_0}(X,Y)=\max_{\lambda_{i_0}+\lambda_{j_0}=1}\{ \lambda_{i_0}\kappa_{i_0}+\lambda_{j_0}\kappa_{j_0}-(\lambda_{i_0}\mu_{i_0}+\lambda_{j_0}\mu_{j_0})(\lambda_{i_0}\nu_{i_0}+\lambda_{j_0}\nu_{j_0}) \}. $$
The optimal solution is given by
$$ \lambda_{i_0}=\frac 12-\frac{\kappa_{j_0}-\mu_{j_0}\nu_{j_0}-\kappa_{i_0}+\mu_{i_0}\nu_{i_0}}{2(\mu_{j_0}-\mu_{i_0})(\nu_{j_0}-\nu_{i_0})}. $$
Actually, we need $0\le\lambda_{i_0}\le1$, which is equivalent to
$$ \lt\vert \frac{\kappa_{j_0}-\mu_{j_0}\nu_{j_0}-\kappa_{i_0}+\mu_{i_0}\nu_{i_0}}{(\mu_{j_0}-\mu_{i_0})(\nu_{j_0}-\nu_{i_0})}\rt\vert \le 1,$$
and equivalent to
$$\min\{\nu_{i_0},\nu_{j_0}\}\le\lt(\frac{\kappa_{j_0}-\mu_{j_0}\nu_{j_0}-\kappa_{i_0}+\mu_{i_0}\nu_{i_0}}{2(\mu_{j_0}-\mu_{i_0})}+\frac{\nu_{i_0}+\nu_{j_0}}{2}\rt)\le\max\{\nu_{i_0}, \nu_{j_0}\}.$$
Then we have
\begin{flalign*}
  V=q_{i_0j_0}(\tilde{\mu}_{i_0j_0}),
\end{flalign*}
where $q_{i_0j_0},\tilde{\mu}_{i_0j_0}$ are defined as in \cref{eq_qij2} and \cref{eq_muij2}.
\par In conclusion, we obtain the desired result.

\end{proof}

\begin{remark}
In Theorem \ref{th1}, we can take $(X,Y)\sim N\left(\mu_i,\nu_i,1,1+|\kappa_i|,\frac{\kappa_i}{\sqrt{1+|\kappa_i|}}\right)=:P_i$ $(1\leq i\leq K)$ to represent $V$ as the upper covariance of such $X$ and $Y$ under $\cp=\{P_i, 1\leq i\leq K\}$.
\end{remark}

\bibliographystyle{plain}
  
\end{document}